\documentclass[english]{article}
\usepackage[utf8]{inputenc}
\usepackage[paperwidth=20cm]{geometry}
\usepackage{rotating}
\usepackage{amsthm}
\usepackage{amsmath}
\usepackage{amssymb}
\usepackage{graphicx}
\usepackage{epsfig}
\usepackage{epic}
\usepackage{anysize}
\usepackage{mathrsfs}
\usepackage{setspace}
\usepackage{rotating}
\usepackage{eso-pic}
\makeatletter
\providecommand{\tabularnewline}{\\}

\newcommand{\lyxaddress}[1]{
\par {\raggedright #1
\vspace{1.4em}
\noindent\par}
}
\theoremstyle{plain}
\newtheorem{thm}{\protect\theoremname}
  \theoremstyle{definition}
  \newtheorem{defn}[thm]{\protect\definitionname}
  \theoremstyle{plain}
  \newtheorem{prop}[thm]{\protect\propositionname}
  \theoremstyle{plain}
  
  \theoremstyle{remark}
  
  \theoremstyle{remark}
  \newtheorem*{rem*}{\protect\remarkname}
  \newtheorem{cor}[thm]{Corollary}
  \newtheorem{example}[thm]{Example}
\makeatother

\newcommand{\vertexcn}{\begin{picture}(20,20)
  \put(6,-2){\line(1,0){20}}\put(6,2){\line(1,0){20}}
  \put(15,0){\line(1,-1){10}}\put(15,0){\line(1,1){10}}
  \end{picture}}
\newcommand{\vertexbn}{\begin{picture}(20,20)
  \put(6,-2){\line(1,0){20}}\put(6,2){\line(1,0){20}}
  \put(20,0){\line(-1,1){10}}\put(20,0){\line(-1,-1){10}}
  \end{picture}}

     \newcommand{\vertexccn}{\begin{picture}(20,20)
  \put(7,0){\line(1,0){20}}
  \put(6,-3){\line(1,0){22}}\put(6,3){\line(1,0){22}}
  \put(15,0){\line(1,1){10}}\put(15,0){\line(1,-1){10}}
  \end{picture}}
\usepackage{babel}
  \providecommand{\definitionname}{Definition}
  \providecommand{\lemmaname}{Lemma}
  \providecommand{\propositionname}{Proposition}
  \providecommand{\remarkname}{Remark}
\providecommand{\theoremname}{Theorem}

\begin{document}

\title{A quick proof of the  classification of  real  Lie superalgebras }

\author{Biswajit Ransingh and K C Pati}

\date{}
\maketitle

\begin{abstract}
This article classifies the real forms of  Lie Superalgebra
by Vogan diagrams, developing Borel and  de Seibenthal theorem of semisimple Lie algebras for Lie superalgebras.
A Vogan diagram is a Dynkin diagram of triplet $(\mathfrak{g}_{C},\mathfrak{h_{\overline{0}}},\triangle^{+})$, where $\mathfrak{g}_{C}$
is a real Lie superalgebra, $\mathfrak{h_{\overline{0}}}$ cartan subalgebra,
$\triangle^{+}$ positive root system. Although the classification of real forms of contragradient Lie superalgebras is already done.
But our method is a quicker one to classify.
\end{abstract}
\noindent 2010 AMS Subject Classification : 17B05, 17B22, 17B40\\
\noindent Keywords : Lie superalgebras, Vogan diagrams
\newline
\section{Introduction}
For a complex semi-simple Lie algebra $\mathfrak{g}$, it is well known that the conjugacy classes of real
forms of $\mathfrak{g}$ are in one to one correspondence with the conjugacy classes of involutions of $\mathfrak{g}$, if
we associate to a real form $\mathfrak{g}_{\mathbb{R}}$ one of its Cartan involutions $\theta$. Using a suited pair $(\mathfrak{h}, \prod)$ of a
Cartan subalgebra $\mathfrak{h}$ and a basis $\prod$ of the associated root system, an involution is described by
a "Vogan diagram".

Knapp \cite{Knapp:Lie groups} brought Vogan diagrams of simple Lie algebras  into the light to  represent 
the real forms of the complex simple Lie algebras.  Batra \cite{batra:affine,batra:vogan} 
developed a corresponding theory of Vogan diagrams for almost compact real forms of
indecomposable nontwisted  affine Kac-Moody Lie algebras.  Similar theory is also developed to find out the Vogan diagrams of hyperbolic 
Kac-Moody algebras \cite{ransingh:pati}.

A Vogan diagram is a Dynkin diagram with some additional 
information as follows. The 2-elements orbits under $\theta$ (Cartan involution) are exhibited 
by joining the corresponding simple roots by a double arrow and the 1-element orbit is painted in
black (respectively, not painted),  if the corresponding imaginary simple root is noncompact (respectively compact).

The real form is defined as a real Lie superalgebra such that its complexification
is the original complex Lie superalgebra. It can be seen easily that every standard real form is
naturally associated to an antilinear involutive automorphism of the complex Lie
superalgebra.

The classification of real semisimple Lie algebras will use maximally compact and split Cartan subalgebras.
 The Vogan diagram is based on the classification of maximally compact Cartan subalgebras.
 Recently similar work has been done using Vogan superdiagrams to classify  the real forms of  
 contragradient Lie superalgebras \cite{chuah:vsuper}, where the extended Dynkin diagrams of Lie superalgebra is used. But
 our method uses the ordinary Dynkin diagram as done by Knapp \cite{Knapp:Lie groups}.
In this article we construct all the real forms of Lie superalgebras by Vogan diagrams.

%A bilinear form $(.,.):\mathfrak{g}\times\mathfrak{g}\rightarrow\mathbb{C}$ on a Lie superalgebra is 
%called \textbf{invariant} if $([x,y],z)=(x,[y,z])$, for all $x,y,z\in \mathfrak{g}$  

%The Lie superalgebra $\mathfrak{g}$ has a root space decomposition with respect to $\mathfrak{h}$ 
%\begin{displaymath} \mathfrak{g}=\mathfrak{h}\oplus\bigoplus_{\alpha\in\triangle}\mathfrak{g}_\alpha \end{displaymath}

%A root $\alpha$ is even if $\mathfrak{g}_{\alpha}\subset\mathfrak{g}_{\overline{0}}$
%and it is odd if $\mathfrak{g}_{\alpha}\subset\mathfrak{g}_{\overline{1}}$ 

%A \textit{Cartan subalgebra} $\mathfrak{h}$ of diagonal matrices
%of $\mathfrak{g}$ is defined to be a Cartan subalgebra of the even
%subalgebra $\mathfrak{g}_{\overline{0}}$. Since every inner automorphism
%of $\mathfrak{g}_{\overline{0}}$ extends to one of Lie superalgebra
%$\mathfrak{g}$ and Cartan subalgebras of $\mathfrak{g}_{\overline{0}}$
%are conjugate under inner automorphisms. So the Cartan subalgebras
%of $\mathfrak{g}$ are conjugate under inner automorphism.
\section{Real forms and Vogan diagrams}
\begin{prop}
[\cite{Parker:classification} Proposition 1.4] Let $\mathcal{\mathfrak{g}}$
be a complex classical Lie superalgebra and let $C$ be an involutive
semimorphism of $\mathcal{\mathfrak{g}}$. Then $\mathfrak{g}_{C}=\left\{ x+Cx|x\in\mathfrak{g}\right\} $
is a real classical Lie superalgebra. 
\end{prop}

\begin{prop}
[\cite{Parker:classification} Proposition 1.5] If $\mathfrak{g}_C$
is a real classical Lie superalgebra, its complexification $\mathfrak{g}=\mathfrak{g}_{C}\otimes\mathbb{C}$
is a Lie superalgebra which is either classical or direct sum of two
isomorphic ideals which are classical.
\end{prop}
\begin{thm}
[\cite{Parker:classification} Theorem 4] Up to isomorphism, the real
forms of the classical Lie superalgebras are uniquely determined by
the real form $\mathfrak{g}_{\overline{0}C}$ of the Lie subalgebra
$\mathcal{\mathfrak{g}}_{\overline{0}}$.
\end{thm}
The real form is said to  be standard (graded) when the real structure
is standard (graded). Let $\mathfrak{g}_{C}$ be a real form
of $\mathfrak{g}$ and let $\omega$ be the corresponding complex
conjugation. Then $\omega|_{\mathfrak{g}_{\overline{0}}}$ is an antilinear involution
of the Lie algebra $\mathfrak{g}_{\overline{0}}$. Hence there is a corresponding
Cartan decomposition $\mathfrak{g}_{\overline{0}}=t_{\overline{0}}\oplus p_{\overline{0}}$.

\begin{defn}
The Vogan diagram of Lie superalgebras is the Dynkin diagram of the basic Lie superalgerbas. In addition to that
\begin{itemize}
 \item[(a)] The vertices fixed by the automorphism (it's Cartan involution for even part) of the 
even part is painted (or unpainted) depending whether the the root is
noncompact (or compact).
 \item[(b)] Label the 2- elements orbit by the diagram automorphism indicated with two sided arrow.
 \item[(c)] The odd root remain unchanged.
\end{itemize}
\end{defn}

We will modify the Borel and de Siebenthal Theorem for Lie superalgebra.
\begin{thm}[Main Theorem]
Let $\mathfrak{g}_{C}$ be a non complex  real Lie superalgebra
and Let the Vogan diagram of $\mathfrak{g}_{C}$ be given that corresponding
to the triple $(\mathfrak{g}_{C},\mathfrak{h}_{0},\triangle^{+})$.
Then $\exists$ a simple root system $\prod'$ for $\triangle=\triangle(\mathcal{\mathfrak{g}},\mathfrak{h})$,
with corresponding positive system $\triangle^{+}$, such that $(\mathcal{\mathfrak{g}}_
{C},\mathfrak{h}_{\overline{0}},\triangle^{+})$
is a triple and there is at most two painted simple root in its Vogan
diagrams of $\mathfrak{sl}(m,n),D(m,n)$ and at most three painted vertices in $D(2,1;\alpha)$. Furthermore suppose the automorphism associated with the
Vogan diagram is the identity,that $\prod'={\alpha_{1},\cdots,\alpha_{l}}$ 
and that ${\omega_{1},\cdots,\omega_{l}}$ is the dual basis for each even part
such that $\left\langle\omega_{j},\alpha_{k}\right\rangle=\delta_{jk}/\epsilon_{kk} $, where $\epsilon_{kk}$ is the diagonal entries to
make Cartan matrix symmetric. Then the double painted simple root of even parts
may be chosen so that there is no $i'$ with $\left\langle \omega_{i}-\omega_{i'},\omega_{i'}\right\rangle>0$ for each even part.

\end{thm}

Proof of Main Theorem.
\begin{proof}
We know $\mathfrak{g}=\mathcal{\mathfrak{g}}_{\bar{0}}\oplus\mathcal{\mathfrak{g}}_{\bar{1}}$. The
positive even root system $\triangle_{0}^{+}$ can be written as 
\[
\triangle_{0}^{+}=\triangle_{01}^{+}\cup\triangle_{02}^{+}
\]
 where $\triangle_{01}^{+}$ are the even positive root system for
simple root system formed by $e_{i}$ basis and $\triangle_{02}^{+}$
are for $\delta_{j}$ basis. For the  even part, we take 
$<\omega_{i},\alpha_{j}>=\delta_{ij}/\epsilon_{kk}$ . This makes the Cartan matrix symmetric and so that 
we can get the inverse of Cartan matrix of $A_{m}$ and $A_n$ for $\mathfrak{sl}(m,n)$ Lie superalgebra. Similarly construction will
follows for other Lie  superalgebras. Each inverse for even part is associated with the dual basis $\omega$.
%For the odd part the condition is  $<\omega_{i},\alpha_{j}>=\delta_{ij}$ and we don't get any painted vertices.
The symmetrizable condition of Kac-Cartan matrix gives
$S=\epsilon_{kk} A$, where $S$ is the symmetric Cartan matrix.
The below table gives the values of $\epsilon_{kk}$ for different superalgebras.

\begin{center}

\begin{tabular}{|c|c|}
\hline 
Lie superalgebra & $\epsilon_{kk}$\tabularnewline
\hline 
$\mathfrak{sl}(m,n)$ & $(1,\cdots,1,-1,\cdots,-1)$\tabularnewline
\hline 
$B(m,n)$ & $(1,\cdots,1,-1\cdots,-1,-2)$\tabularnewline
\hline 
$B(0,n)$ & $1,\cdots,1,2$\tabularnewline
\hline 
$C(n)$ & $(-1,1,\cdots,1,\frac{1}{2})$\tabularnewline
\hline 
$D(m,n)$ & $(1,\cdots,1,-1-1,\cdots,-1,-1,-1)$\tabularnewline
\hline 
$D(2,1;\alpha)$ & $(1,-1,\frac{1}{a})$\tabularnewline
\hline 
$F(4)$ & $ (-1,1,\frac{1}{2}) $\tabularnewline
\hline 
$G(3)$ & $(-\frac{1}{2},1,\frac{1}{3},)$\tabularnewline
\hline 
\end{tabular}             \end{center}
Taking suitable  normalization condition for each type of Lie superalgebras and 
 from the two Lemmas 6.97 and 6.98 \cite{Knapp:Lie groups} we get redudancy test for each even part.
So now the Vogan diagram of Lie superalgebras becomes two painted
vertices Vogan diagram for $\mathfrak{sl}(m,n)$, $D(m,n)$ and three painted for $D(2,1;\alpha)$.\end{proof}
In fact from the results of Main theorem the following Corollary is immediate.
\begin{cor}
 The number of painted vertices of the basic real Lie superalgebra by Vogan diagram
 is number of even parts of the corresponding Lie superalgebra.
\end{cor}
\subsection{Equivalence of Vogan diagrams}

 A reflection change the colors of neighbor  to opposite color except the odd root
 and change neighbor color to opposite if it is a shorter root.
 \begin{example} The two Vogan diagrams of $C(m)$ are equivalence by above rule.
  
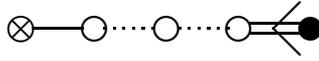
\begin{displaymath}
       \begin{picture}(60,20) \thicklines 
\put(-48,-3){$\bigotimes$} \put(-15,0){\circle{9}}\put(12,0){\circle{9}} \put(39,0){\circle{9}}\put(66,0){\circle*{9}}
 \put(-38,0){\line(1,0){19}}\put(-12,0){\dottedline{4}(1,0)(19,0)}\put(16,0){\dottedline{4}(1,0)(19,0)}\put(37,0){\vertexcn}
   \end{picture}
\end{displaymath}
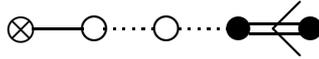
\vspace{0.5cm}
 \begin{displaymath}
       \begin{picture}(60,20) \thicklines 
\put(-48,-3){$\bigotimes$} \put(-15,0){\circle{9}}\put(12,0){\circle{9}} \put(39,0){\circle*{9}}\put(66,0){\circle*{9}}
 \put(-38,0){\line(1,0){19}}\put(-12,0){\dottedline{4}(1,0)(19,0)}\put(16,0){\dottedline{4}(1,0)(19,0)}\put(37,0){\vertexcn}
   \end{picture}
\end{displaymath}
\end{example}
 
\begin{enumerate}
\item  $\mathfrak{sl}(m,n)$

The Vogan diagrams and real foms of Lie superalgebras $A(m,n)$ are as follows.

  
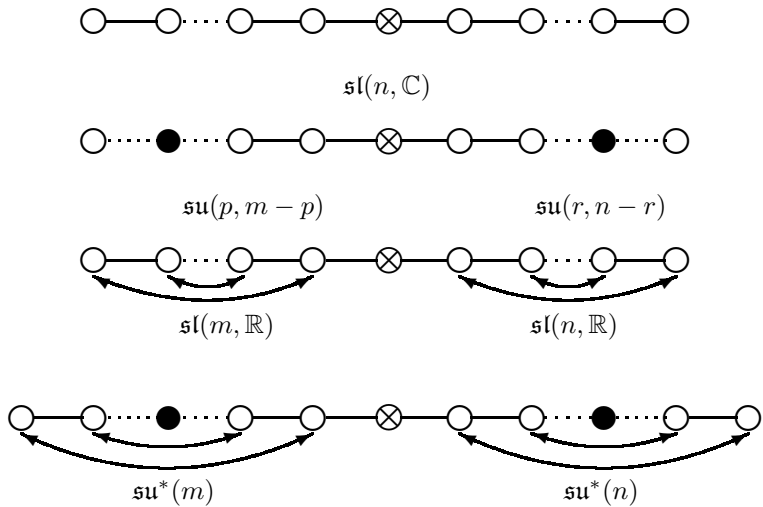
\begin{displaymath}
       \begin{picture}(60,10) \thicklines 
\put(-125,0){\circle{9}} \put(-97,0){\circle{9}}\put(-70,0){\circle{9}} \put(-43,0){\circle{9}} \put(-20,-2.5){$\bigotimes$}\put(12,0){\circle{9}} 
 \put(39,0){\circle{9}}\put(66,0){\circle{9}}\put(93,0){\circle{9}}
 \put(-120,0){\line(1,0){19}} \put(-92,0){\dottedline{4}(1,0)(17,0)}\put(-65,0){\line(1,0){17}}\put(-38,0){\line(1,0){19}}
 \put(-10,0){\line(1,0){17}}
 \put(16,0){\line(1,0){18}} \put(43,0){\dottedline{4}(1,0)(18,0)}\put(70,0){\line(1,0){19}}
 \put(-15,-25){\makebox(0,0){$\mathfrak{sl}(n,\mathbb{C})$}}
   \end{picture}
\end{displaymath}

\vspace{0.5cm}
  \begin{displaymath}
       \begin{picture}(60,10) \thicklines 
\put(-125,0){\circle{9}} \put(-97,0){\circle*{9}}\put(-70,0){\circle{9}} \put(-43,0){\circle{9}} \put(-20,-2.5){$\bigotimes$}\put(12,0){\circle{9}} 
 \put(39,0){\circle{9}}\put(66,0){\circle*{9}}\put(93,0){\circle{9}}
 \put(-120,0){\dottedline{4}(1,0)(19,0)} \put(-92,0){\dottedline{4}(1,0)(17,0)}\put(-65,0){\line(1,0){17}}\put(-38,0){\line(1,0){19}}
 \put(-10,0){\line(1,0){17}}
 \put(16,0){\line(1,0){18}} \put(43,0){\dottedline{4}(1,0)(18,0)}\put(70,0){\dottedline{4}(1,0)(19,0)}
 \put(-65,-25){\makebox(0,0){$\mathfrak{su}(p,m-p)$}}\put(65,-25){\makebox(0,0){$\mathfrak{su}(r,n-r)$}}
   \end{picture}
\end{displaymath}

\vspace{0.5cm}

  \begin{displaymath}
       \begin{picture}(60,10) \thicklines 
\put(-125,0){\circle{9}} \put(-97,0){\circle{9}}\put(-70,0){\circle{9}} \put(-43,0){\circle{9}}
\put(-20,-2.5){$\bigotimes$}\put(12,0){\circle{9}} 
 \put(39,0){\circle{9}}\put(66,0){\circle{9}}\put(93,0){\circle{9}}
 \put(-120,0){\line(1,0){19}} \put(-92,0){\dottedline{4}(1,0)(17,0)}\put(-65,0){\line(1,0){17}}\put(-38,0){\line(1,0){19}}
 \put(-10,0){\line(1,0){17}}
 \put(16,0){\line(1,0){18}} \put(43,0){\dottedline{4}(1,0)(18,0)}\put(70,0){\line(1,0){19}}
  \qbezier(-97,-6)(-81.5,-14)(-69, -6)
   \put(-97,-6){\vector( -2, 1){0}}\put(-69,-6){\vector( 2, 1){0}}
     \qbezier(-125,-6)(-81,-24)(-43, -6)
   \put(-125,-6){\vector( -2, 1){0}}\put(-43,-6){\vector( 2, 1){0}}
   \qbezier(39,-6)(52.5,-14)(66, -6)
   \put(39,-6){\vector( -2, 1){0}}\put(66,-6){\vector( 2, 1){0}}
     \qbezier(12,-6)(52.5,-24)(93, -6)
   \put(12,-6){\vector( -2, 1){0}}\put(93,-6){\vector( 2, 1){0}}
    \put(-75,-25){\makebox(0,0){$\mathfrak{sl}(m,\mathbb{R})$}}\put(55,-25){\makebox(0,0){$\mathfrak{sl}(n,\mathbb{R})$}}
   \end{picture}
\end{displaymath}

\vspace{1cm}
 \begin{displaymath}
       \begin{picture}(60,10) \thicklines 
\put(-152,0){\circle{9}}\put(-125,0){\circle{9}} \put(-97,0){\circle*{9}}\put(-70,0){\circle{9}} \put(-43,0){\circle{9}}
\put(-20,-2.5){$\bigotimes$}\put(12,0){\circle{9}} 
 \put(39,0){\circle{9}}\put(66,0){\circle*{9}}\put(93,0){\circle{9}}\put(120,0){\circle{9}}
 \put(-147,0){\line(1,0){18}}\put(-120,0){\dottedline{4}(1,0)(19,0)} \put(-92,0){\dottedline{4}(1,0)(17,0)}
 \put(-65,0){\line(1,0){17}}\put(-38,0){\line(1,0){19}}
  \put(-10,0){\line(1,0){17}}
 \put(16,0){\line(1,0){18}} \put(43,0){\dottedline{4}(1,0)(18,0)}\put(70,0){\dottedline{4}(1,0)(19,0)} \put(98,0){\line(1,0){18}}
  \qbezier(-125,-6)(-97.5,-16)(-70, -6)
   \put(-125,-6){\vector( -3, 1){0}}\put(-70,-6){\vector( 3, 1){0}}
    \qbezier(-152,-6)(-97.5,-32)(-43, -6)
   \put(-152,-6){\vector( -2, 1){0}}\put(-43,-6){\vector(2, 1){0}}
   
   \qbezier(39,-6)(66,-16)(93, -6)
   \put(39,-6){\vector( -3, 1){0}}\put(93,-6){\vector( 3, 1){0}}
    \qbezier(12,-6)(66,-32)(120, -6)
   \put(12,-6){\vector( -2, 1){0}}\put(120,-6){\vector(2, 1){0}}
 \put(-95,-28){\makebox(0,0){$\mathfrak{su}^{*}(m)$}}\put(65,-28){\makebox(0,0){$\mathfrak{su}^{*}(n)$}}
   \end{picture}
  \end{displaymath}
  
     \vspace{1cm}
\item $B(m,n)$.
The Lie subalgebra of $\mathfrak{g}_{0}$ is $C_{m}\oplus B_{n}$
The only trivial automophism of even part of  Vogan diagram of $B(m,n)$ is shown
below and the real form is $\mathfrak{sp}(2n,\mathbb{R})\oplus \mathfrak{so}(p,q)$

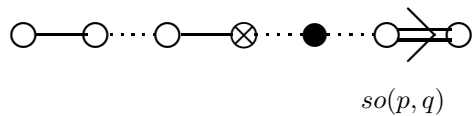
  \begin{displaymath}
       \begin{picture}(60,10) \thicklines 
 \put(-97,0){\circle{9}}\put(-70,0){\circle{9}} \put(-43,0){\circle{9}} \put(-20,-2.5){$\bigotimes$}\put(12,0){\circle*{9}} 
 \put(39,0){\circle{9}}\put(66,0){\circle{9}}
 \put(-92,0){\line(1,0){19}}\put(-65,0){\dottedline{4}(1,0)(17,0)}\put(-38,0){\line(1,0){19}}\put(-11,0){\dottedline{4}(1,0)(18,0)}
 \put(16,0){\dottedline{4}(1,0)(18,0)}\put(37,0){\vertexbn}
\put(45,-25){\makebox(0,0){$so(p,q)$}}
   \end{picture}
\end{displaymath}

\vspace{1cm}
Because of missing of real form of the first even part, we need an additional $C_{n}$ Dynkin diagram superimposed 
Vogan diagrams below.
  
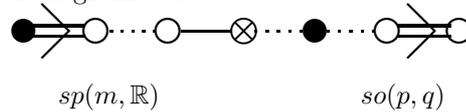
\begin{displaymath}
       \begin{picture}(60,10) \thicklines 
 \put(-97,0){\circle*{9}}\put(-70,0){\circle{9}} \put(-43,0){\circle{9}} \put(-20,-2.5){$\bigotimes$}\put(12,0){\circle*{9}} 
 \put(39,0){\circle{9}}\put(66,0){\circle{9}}
\put(-65,0){\dottedline{4}(1,0)(17,0)}\put(-38,0){\line(1,0){19}}\put(-11,0){\dottedline{4}(1,0)(18,0)}
 \put(16,0){\dottedline{4}(1,0)(18,0)}\put(37,0){\vertexbn}
 \put(-65,-25){\makebox(0,0){$sp(m,\mathbb{R})$}}\put(45,-25){\makebox(0,0){$so(p,q)$}}
 \put(-100,0)\vertexbn
   \end{picture}
\end{displaymath}

\vspace{1cm}

\item $B(0,n)$. The  Vogan diagram below is a unpainted diagram but it consists of its own painted vertices on the extreme right.\\

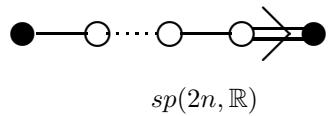
  \begin{displaymath}
       \begin{picture}(60,20) \thicklines 
\put(-43,0){\circle*{9}} \put(-15,0){\circle{9}}\put(12,0){\circle{9}} \put(39,0){\circle{9}}\put(66,0){\circle*{9}}
 \put(-38,0){\line(1,0){19}}\put(-12,0){\dottedline{4}(1,0)(19,0)}\put(16,0){\line(1,0){19}}\put(37,0){\vertexbn}
  \put(25,-25){\makebox(0,0){$sp(2n,\mathbb{R})$}}
   \end{picture}
\end{displaymath}
\vspace{1cm}
\item  $C(m)$.

The unpainted Vogan diagram of $C(m)$ correspondence to  the real form  $so(2)\oplus sp(m,\mathbb{R})$\\

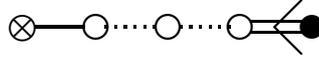
     \begin{displaymath}
       \begin{picture}(60,20) \thicklines 
\put(-48,-3){$\bigotimes$} \put(-15,0){\circle{9}}\put(12,0){\circle{9}} \put(39,0){\circle{9}}\put(66,0){\circle*{9}}
 \put(-38,0){\line(1,0){19}}\put(-12,0){\dottedline{4}(1,0)(19,0)}\put(16,0){\dottedline{4}(1,0)(19,0)}\put(37,0){\vertexcn}
   \end{picture}
\end{displaymath}
\vspace{0.5cm}

The trivial automophism of the even part of $C(m)$ makes the Vogan diagram
below  and the real form is $so(2)\oplus sp(r,s)$\\

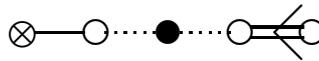
     \begin{displaymath}
       \begin{picture}(60,20) \thicklines 
\put(-48,-3){$\bigotimes$} \put(-15,0){\circle{9}}\put(12,0){\circle*{9}} \put(39,0){\circle{9}}\put(66,0){\circle{9}}
 \put(-38,0){\line(1,0){19}}\put(-12,0){\dottedline{4}(1,0)(19,0)}\put(16,0){\dottedline{4}(1,0)(19,0)}\put(37,0){\vertexcn}
   \end{picture}
\end{displaymath}
\vspace{0.5cm}
\item  $D(m,n)$. 
The Lie subalgebra of $\mathfrak{g}_{0}$ is  $C_{m}\oplus D_{n}$ .

The  trivial involution for the Vogan diagram of $D(m,n)$
is given below and the real form of this diagram is $sp(r,s)\oplus so^{*}(2p)$.

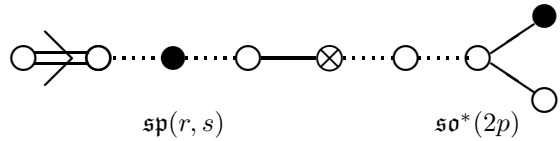
\begin{displaymath}
   \begin{picture}(-20,20)\thicklines
   \put(-165,0){\circle{9}}
 \put(-137,0){\circle{9}} \put(-137,0){\circle{9}}\put(-109,0){\circle*{9}}\put(-81,0){\circle{9}}\put(-56,-2.5){$\bigotimes$} \put(-22,0){\circle{9}} 
  \put(5,0){\circle{9}}\put(30,16){\circle*{9}} \put(30,-16){\circle{9}} 
 \put(-132,0){\dottedline{4}(1,0)(19,0)} \put(-105,0){\dottedline{4}(1,0)(19,0)}
 \put(-76,0){\line(1,0){20}} \put(-46,0){\dottedline{4}(1,0)(19,0)}\put(-18,0){\dottedline{4}(1,0)(19,0)} \put(10,2){\line(3,2){16}}\put(10,-2){\line(3,-2){16}}
  
   \put(-105,-25){\makebox(0,0){$\mathfrak{sp}(r,s)$}}\put(5,-25){\makebox(0,0){$\mathfrak{so}^{*}(2p)$}}
      \put(-167,0)\vertexbn
     \end{picture}
   \end{displaymath}
\vspace{1cm}

The below Vogan diagram is formed by the nontrivial and trivial involution with associated
real form is the same $\mathfrak{sp}(m,\mathbb{R})\oplus \mathfrak{so}(p,q)$.

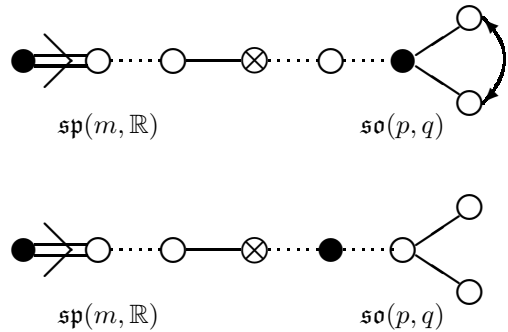
 \begin{displaymath}
   \begin{picture}(-20,20)\thicklines
  \put(-137,0){\circle*{9}}\put(-109,0){\circle{9}}\put(-81,0){\circle{9}}\put(-56,-2.5){$\bigotimes$} \put(-22,0){\circle{9}} 
  \put(5,0){\circle*{9}}\put(30,16){\circle{9}} \put(30,-16){\circle{9}} 
 \put(-105,0){\dottedline{4}(1,0)(19,0)}
 \put(-76,0){\line(1,0){20}} \put(-46,0){\dottedline{4}(1,0)(19,0)}\put(-18,0){\dottedline{4}(1,0)(19,0)} \put(10,2){\line(3,2){16}}\put(10,-2){\line(3,-2){16}}
   \qbezier(35,16)(51,1)(35, -16)
   \put(35,16){\vector( -1, 1){0}}\put(35,-16){\vector( -1, -1){0}}
   \put(-105,-25){\makebox(0,0){$\mathfrak{sp}(m,\mathbb{R})$}}\put(5,-25){\makebox(0,0){$\mathfrak{so}(p,q)$}}
   \put(-139,0)\vertexbn
     \end{picture}
   \end{displaymath}
   
\vspace{1cm} 
\begin{displaymath}
   \begin{picture}(-20,20)\thicklines
  \put(-137,0){\circle*{9}}\put(-109,0){\circle{9}}\put(-81,0){\circle{9}}\put(-56,-2.5){$\bigotimes$} \put(-22,0){\circle*{9}} 
  \put(5,0){\circle{9}}\put(30,16){\circle{9}} \put(30,-16){\circle{9}} 
  \put(-105,0){\dottedline{4}(1,0)(19,0)}
 \put(-76,0){\line(1,0){20}} \put(-46,0){\dottedline{4}(1,0)(19,0)}\put(-18,0){\dottedline{4}(1,0)(19,0)}
 \put(10,2){\line(3,2){16}}\put(10,-2){\line(3,-2){16}}
  \put(-105,-25){\makebox(0,0){$\mathfrak{sp}(m,\mathbb{R})$}}\put(5,-25){\makebox(0,0){$\mathfrak{so}(p,q)$}}
   \put(-139,0)\vertexbn
       \end{picture}
 \end{displaymath}
\vspace{1cm}
\item  $D(2,1;\alpha)$.

The unpainted and no two element orbit  Vogan diagram is given  below  with the real form. 
Since we can get only the real form $\mathfrak{su}(2)\oplus \mathfrak{su}(2)$ from ordinary Vogan  diagram 
       
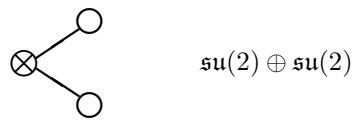
\begin{displaymath}
   \begin{picture}(40,20)\thicklines
\put(0,-2.5){$\bigotimes$}\put(30,16){\circle{9}} \put(30,-16){\circle{9}} 
    \put(10,2){\line(3,2){16}}\put(10,-2){\line(3,-2){16}}
       \put(100,0){\makebox(0,0){$\mathfrak{su}(2)\oplus \mathfrak{su}(2)$}}
  \end{picture}
   \end{displaymath}
    \vspace{0.5cm}
    
    So our requisite Vogan diagrams for the suitable real forms are
   
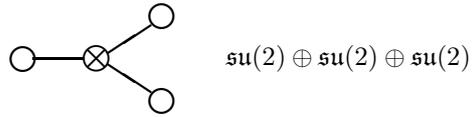
\begin{displaymath}
   \begin{picture}(40,20)\thicklines
 \put(-22,0){\circle{9}} \put(0,-2.5){$\bigotimes$}\put(30,16){\circle{9}} \put(30,-16){\circle{9}} 
   \put(-18,0){\line(1,0){19}} \put(10,2){\line(3,2){16}}\put(10,-2){\line(3,-2){16}}
       \put(100,0){\makebox(0,0){$\mathfrak{su}(2)\oplus \mathfrak{su}(2)\oplus \mathfrak{su}(2)$}}
  \end{picture}
   \end{displaymath}
    
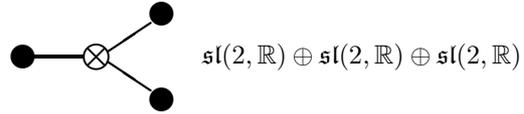
\vspace{1cm}
 \begin{displaymath}
   \begin{picture}(40,20)\thicklines
 \put(-22,0){\circle*{9}}  \put(0,-2.5){$\bigotimes$}\put(30,16){\circle*{9}} \put(30,-16){\circle*{9}} 
   \put(-18,0){\line(1,0){19}}  \put(10,2){\line(3,2){16}}\put(10,-2){\line(3,-2){16}}
              \put(105,0){\makebox(0,0){$\mathfrak{sl}(2,\mathbb {R})\oplus\mathfrak{ sl}(2,\mathbb {R})\oplus
              \mathfrak{sl}(2,\mathbb {R})$}}
  \end{picture}
   \end{displaymath}
\vspace{1cm}
    
The nontrivial involution of the Dynkin diagram of $D(2,1;\alpha)$
makes the following Vogan diagram as shown below.
    
   
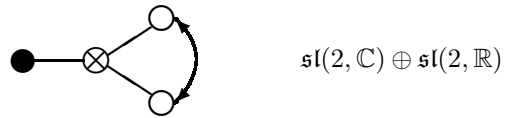
\begin{displaymath}
   \begin{picture}(40,20)\thicklines
 \put(-22,0){\circle*{9}}  \put(0,-2.5){$\bigotimes$}\put(30,16){\circle{9}} \put(30,-16){\circle{9}} 
 \put(-18,0){\line(1,0){19}} \put(10,2){\line(3,2){16}}\put(10,-2){\line(3,-2){16}}
   \qbezier(35,16)(51,1)(35, -16)
   \put(35,16){\vector( -1, 1){0}}\put(35,-16){\vector( -1, -1){0}}
   \put(120,0){\makebox(0,0){$\mathfrak{sl}(2,\mathbb{C})\oplus \mathfrak{sl}(2,\mathbb{R})$}}
  \end{picture}
   \end{displaymath}

\vspace{1cm}

\item  F(4).
From the Dynkin diagram we can get only  the real form $so(7)$ and the Vogan diagram 

     
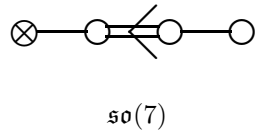
\begin{displaymath}
       \begin{picture}(60,20) \thicklines 
\put(-48,-3){$\bigotimes$} \put(-15,0){\circle{9}}\put(12,0){\circle{9}} \put(39,0){\circle{9}}
 \put(-38,0){\line(1,0){19}}\put(-18,0){\vertexcn}\put(16,0){\line(1,0){19}}
 \put(0,-32){\makebox(0,0){$\mathfrak{so}(7)$}} 
   \end{picture}
\end{displaymath}
\vspace{1cm}

we add  the extra even part root to get the desired  real forms and Vogan diagrams. 
     
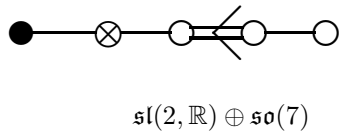
\begin{displaymath}
       \begin{picture}(60,20) \thicklines 
\put(-75,0){\circle*{9}}\put(-48,-3){$\bigotimes$} \put(-15,0){\circle{9}}\put(12,0){\circle{9}} \put(39,0){\circle{9}}
\put(-71,0){\line(1,0){24}} \put(-38,0){\line(1,0){19}}\put(-18,0){\vertexcn}\put(16,0){\line(1,0){19}}
 \put(0,-32){\makebox(0,0){$\mathfrak{sl}(2,\mathbb{R})\oplus \mathfrak{so}(7)$}} 
   \end{picture}
\end{displaymath}
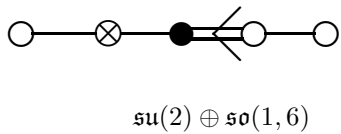
\vspace{1cm}
   \begin{displaymath}
    \begin{picture}(60,20) \thicklines 
\put(-75,0){\circle{9}}\put(-48,-2.5){$\bigotimes$} \put(-15,0){\circle*{9}}\put(12,0){\circle{9}} \put(39,0){\circle{9}}
 \put(-71,0){\line(1,0){24}} \put(-38,0){\line(1,0){19}}\put(-18,0){\vertexcn}\put(16,0){\line(1,0){19}}
 \put(0,-32){\makebox(0,0){$\mathfrak{su}(2)\oplus \mathfrak{so}(1,6)$}} 
   \end{picture}
\end{displaymath}
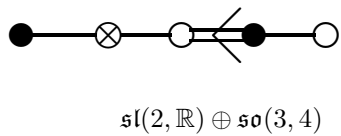
\vspace{1cm}
   \begin{displaymath}
    \begin{picture}(60,20) \thicklines 
\put(-75,0){\circle*{9}}\put(-48,-2.5){$\bigotimes$} \put(-15,0){\circle{9}}\put(12,0){\circle*{9}} \put(39,0){\circle{9}}
  \put(-71,0){\line(1,0){24}}\put(-38,0){\line(1,0){19}}\put(-18,0){\vertexcn}\put(16,0){\line(1,0){19}}
  \put(0,-32){\makebox(0,0){$\mathfrak{sl}(2,\mathbb{R})\oplus \mathfrak{so}(3,4)$}} 
   \end{picture}
\end{displaymath}
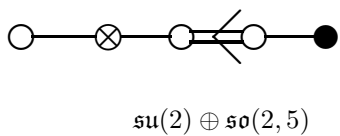
\vspace{1cm}
   \begin{displaymath}
    \begin{picture}(60,20) \thicklines 
\put(-75,0){\circle{9}} \put(-48,-2.5){$\bigotimes$} \put(-15,0){\circle{9}}\put(12,0){\circle{9}} \put(39,0){\circle*{9}}
  \put(-71,0){\line(1,0){24}} \put(-38,0){\line(1,0){19}}\put(-18,0){\vertexcn}\put(16,0){\line(1,0){19}}
  \put(0,-32){\makebox(0,0){$\mathfrak{su}(2)\oplus \mathfrak{so}(2,5)$}} 
   \end{picture}
\end{displaymath}
\vspace{1cm}

\item $G(3)$.

The Vogan diagram of $G(3)$ with real form  $\mathfrak{sl}(2,\mathbb{R})\oplus G_{2,0}$ and $\mathfrak{sl}(2,\mathbb{R})\oplus G_{2,2}$
are

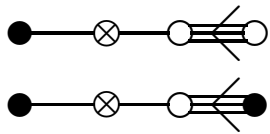
\begin{displaymath}
\begin{picture}(60,20) \thicklines
\put(-34,0){\circle*{9}}\put(-7,-2.5){$\bigotimes$}\put(26,0){\circle{9}}\put(54,0){\circle{9}}
\put(-30,0){\line(1,0){24}}\put(3,0){\line(1,0){19}}\put(23,0){\vertexccn}
\end{picture}
\end{displaymath}
\begin{displaymath}
\begin{picture}(60,20) \thicklines
\put(-34,0){\circle*{9}}\put(-7,-2.5){$\bigotimes$}\put(26,0){\circle{9}}\put(54,0){\circle*{9}}
\put(-30,0){\line(1,0){24}}\put(3,0){\line(1,0){19}}\put(23,0){\vertexccn}
\end{picture}
\end{displaymath}
\end{enumerate}

\paragraph{Acknowledgement:}

The authors thank National Board of Higher Mathematics, India (Project
Grant No. 48/3/2008-R\&DII/196-R) for financial support.

\lyxaddress{\begin{center}
Biswajit Ransingh, Harish Chandra Research Institute, Allahabad(INDIA) 211019, email- bransingh@gmail.com 
\end{center}}
\lyxaddress{\begin{center}
K C Pati, Department of Mathematics, National Institute of Technology, Rourkela (India) 769008, email- kcpati@rediffmail.com 
\end{center}}

\begin{thebibliography}{9}
\bibitem{knapp1} Knapp A.W., {A quick proof  of the classification of simple real Lie algebra}, Proc. American Math. Society, 124(10) (1996).
\bibitem{batra:affine} Batra P, {\em Invariant of Real forms of Affine Kac-Moody Lie algebras}, Journal of Algebra 223, 
208-236 (2000).
\bibitem{batra:vogan} Batra P, {\em Vogan diagrams of affine Kac-Moody algebras}, Journal of Algebra 251, 80-97 (2002). 

\bibitem{Knapp:Lie groups}Knapp A.W., {\em Lie groups beyond an Introduction}, Second Edition.
\bibitem{chuah:vsuper}Meng-Kiat Chuah, {\em Cartan automorphisms and Vogan superdiagrams}, Math.Z. 273, 793-800 (2012).
\bibitem{pati:ssuper}Pati  K.C., Parashar, D. {\em Satake superdiagrams, real forms and Iwasawa decomposition of classical Lie
superalgebras}, J. Phys. A 31, 767–778 (1998)
\bibitem{Parker:classification} Parker M, {\em Classification of real simple Lie superalgebras of classical type}, 
J. Math.Phys. 21(4), April (1980).
\bibitem{ransingh:pati}Ransingh B and Pati K.C., {\em Vogan diagrams of some hyperbolic Kac-Moody algebras},  arXiv:1205.3724 [math.RT]  (2012). 
%\bibitem{tanushree:twisted} Paul Tanushree, {\em Vogan diagram of twisted Affine Kac-Moody lie algebras}, 
%Pacific Jouranl of mathematics, 239 (1), January (2009).
%\bibitem{grothen:LSA}Sergeev Alexander N. and Veselov Alexander P., {\em Grothendieck rings of basic 
%classical Lie superalgebras}, Annals of Mathematics 173 (2) 663 - 703 (2011). 
\end{thebibliography}
\end{document}